\documentclass[10pt,a4paper]{article}
\usepackage[utf8]{inputenc}
\usepackage[english]{babel}
\usepackage[T1]{fontenc}

\usepackage{amsmath}
\usepackage{amsthm}
\usepackage{amsfonts}
\usepackage{amssymb}
\usepackage{scrextend}
\usepackage[colorlinks=true, allcolors=blue]{hyperref}
\usepackage{cleveref}
\usepackage{graphicx}
\usepackage{enumerate}
\usepackage{lmodern}
\usepackage{wrapfig}
\usepackage{mathrsfs}
\usepackage[labelsep=period]{caption}
\usepackage{subcaption}
\usepackage[onelanguage,ruled,vlined,linesnumbered]{algorithm2e}
\usepackage{siunitx}
\usepackage{pdfpages}
\usepackage{float}
\usepackage[colorinlistoftodos]{todonotes}
\usepackage{comment}
\usepackage[parfill]{parskip}
\usepackage{bm}
\usepackage{xcolor}
\usepackage{colortbl}
\usepackage{makecell}
\usepackage{tablefootnote}

\usepackage{diagbox}

\usepackage{tikz}
\usetikzlibrary{shapes,graphs,graphs.standard}

\newcommand{\db}[2]{#1\!\leftrightarrow\!#2}
\newcommand{\ru}[1]{\textbf{R#1}}

\newtheorem{theorem}{Theorem}
\newtheorem{corollary}[theorem]{Corollary}
\newtheorem{definition}[theorem]{Definition}

\newtheorem{proposition}[theorem]{Proposition}
\crefname{proposition}{Proposition}{Propositions}
\newtheorem{lemma}[theorem]{Lemma}
\newtheorem{conjecture}[theorem]{Conjecture}

\crefname{claim}{claim}{claims}
\Crefname{claim}{Claim}{Claims}

\usepackage[left=2cm,right=2cm,top=2cm,bottom=2cm]{geometry}

\renewcommand{\thesubfigure}{\roman{subfigure}}
\DeclareMathOperator{\ad}{ad}
\DeclareMathOperator{\mad}{mad}

\usepackage{authblk}

\title{$2$-distance $(\Delta+2)$-coloring of sparse graphs}

\author[1]{Hoang La\thanks{xuan-hoang.la@lirmm.fr}}
\author[1]{Mickael Montassier\thanks{mickael.montassier@lirmm.fr}}
\affil[1]{LIRMM, Université de Montpellier, CNRS, Montpellier, France}
\begin{document}
  \maketitle

\begin{abstract}
A $2$-distance $k$-coloring of a graph is a proper $k$-coloring of the vertices where vertices at distance at most 2 cannot share the same color. We prove the existence of a $2$-distance ($\Delta+2$)-coloring for graphs with maximum average degree less than $\frac{8}{3}$ (resp. $\frac{14}{5}$) and maximum degree $\Delta\geq 6$ (resp. $\Delta\geq 10$). As a corollary, every planar graph with girth at least $8$ (resp. $7$) and maximum degree $\Delta\geq 6$ (resp. $\Delta\geq 10$) admits a $2$-distance $(\Delta+2)$-coloring.
\end{abstract}

\section{Introduction}

A \emph{$k$-coloring} of the vertices of a graph $G=(V,E)$ is a map $\phi:V \rightarrow\{1,2,\dots,k\}$. A $k$-coloring $\phi$ is a \emph{proper coloring}, if and only if, for all edge $xy\in E,\phi(x)\neq\phi(y)$. In other words, no two adjacent vertices share the same color. The \emph{chromatic number} of $G$, denoted by $\chi(G)$, is the smallest integer $k$ such that $G$ has a proper $k$-coloring.  A generalization of $k$-coloring is $k$-list-coloring.
A graph $G$ is {\em $L$-list colorable} if for a
given list assignment $L=\{L(v): v\in V(G)\}$ there is a proper
coloring $\phi$ of $G$ such that for all $v \in V(G), \phi(v)\in
L(v)$. If $G$ is $L$-list colorable for every list assignment $L$ with $|L(v)|\ge k$ for all $v\in V(G)$, then $G$ is said to be {\em $k$-choosable} or \emph{$k$-list-colorable}. The \emph{list chromatic number} of a graph $G$ is the smallest integer $k$ such that $G$ is $k$-choosable. List coloring can be very different from usual coloring as there exist graphs with a small chromatic number and an arbitrarily large list chromatic number.

In 1969, Kramer and Kramer introduced the notion of 2-distance coloring \cite{kramer2,kramer1}. This notion generalizes the ``proper'' constraint (that does not allow two adjacent vertices to have the same color) in the following way: a \emph{$2$-distance $k$-coloring} is such that no pair of vertices at distance at most 2 have the same color. The \emph{$2$-distance chromatic number} of $G$, denoted by $\chi^2(G)$, is the smallest integer $k$ such that $G$ has a 2-distance $k$-coloring. Similarly to proper $k$-list-coloring, one can also define \emph{$2$-distance $k$-list-coloring} and a \emph{$2$-distance list chromatic number}.

For all $v\in V$, we denote $d_G(v)$ the degree of $v$ in $G$ and by $\Delta(G) = \max_{v\in V}d_G(v)$ the maximum degree of a graph $G$. For brevity, when it is clear from the context, we will use $\Delta$ (resp. $d(v)$) instead of $\Delta(G)$ (resp. $d_G(v)$). 
One can observe that, for any graph $G$, $\Delta+1\leq\chi^2(G)\leq \Delta^2+1$. The lower bound is trivial since, in a 2-distance coloring, every neighbor of a vertex $v$ with degree $\Delta$, and $v$ itself must have a different color. As for the upper bound, a greedy algorithm shows that $\chi^2(G)\leq \Delta^2+1$. Moreover, that upper bound is tight for some graphs, for example, Moore graphs of type $(\Delta,2)$, which are graphs where all vertices have degree $\Delta$, are at distance at most two from each other, and the total number of vertices is $\Delta^2+1$. See \Cref{tight upper bound figure}.

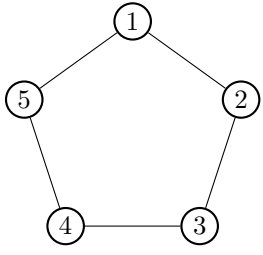
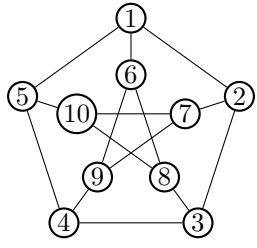
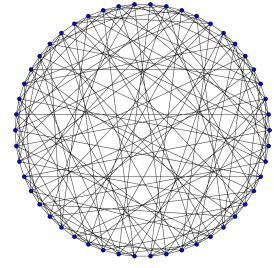
\begin{figure}[htbp]
\begin{center}
\begin{subfigure}[t]{5cm}
\centering
\begin{tikzpicture}[every node/.style={circle,thick,draw,minimum size=1pt,inner sep=2}]
  \graph[clockwise, radius=1.5cm] {subgraph C_n [n=5,name=A] };
\end{tikzpicture}
\caption{The Moore graph of type (2,2):\\ the odd cycle $C_5$.}
\end{subfigure}
\qquad
\begin{subfigure}[t]{5cm}
\centering
\begin{tikzpicture}[every node/.style={circle,thick,draw,minimum size=1pt,inner sep=1}]
  \graph[clockwise, radius=1.5cm] {subgraph C_n [n=5,name=A] };
  \graph[clockwise, radius=0.75cm,n=5,name=B] {1/"6", 2/"7", 3/"8", 4/"9", 5/"10" };

  \foreach \i [evaluate={\j=int(mod(\i+6,5)+1)}]
     in {1,2,3,4,5}{
    \draw (A \i) -- (B \i);
    \draw (B \j) -- (B \i);
  }
\end{tikzpicture}
\caption{The Moore graph of type (3,2):\\ the Petersen graph.}
\end{subfigure}
\qquad
\begin{subfigure}[t]{5cm}
\centering
\includegraphics[scale=0.12]{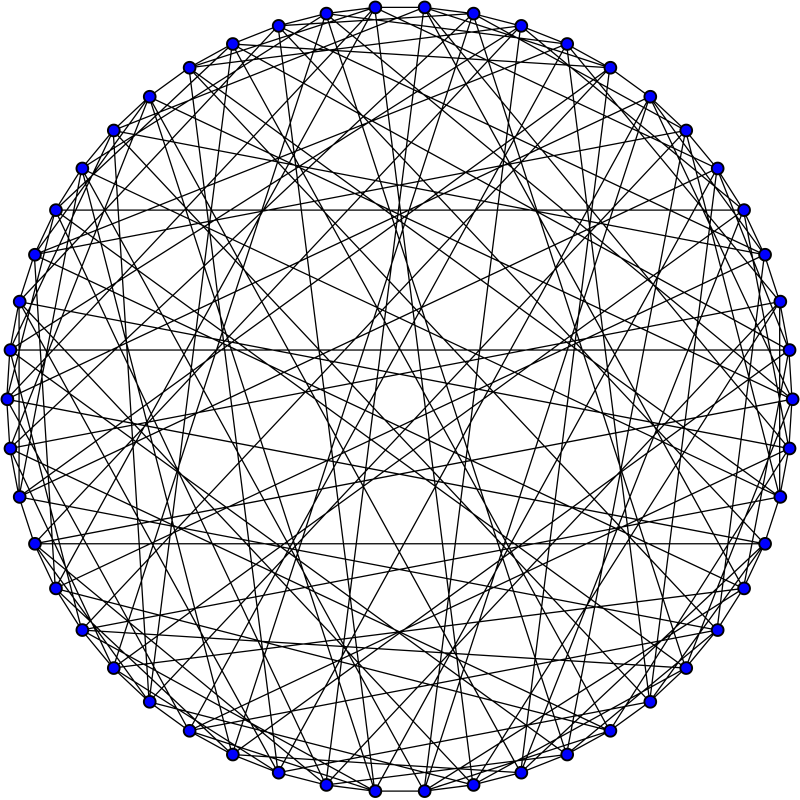}
\caption{The Moore graph of type (7,2):\\ the Hoffman-Singleton graph.}
\end{subfigure}
\caption{Examples of Moore graphs for which $\chi^2=\Delta^2+1$.}
\label{tight upper bound figure}
\end{center}
\end{figure}

By nature, $2$-distance colorings and the $2$-distance chromatic number of a graph depend a lot on the number of vertices in the neighborhood of every vertex. More precisely, the ``sparser'' a graph is, the lower its $2$-distance chromatic number will be. One way to quantify the sparsity of a graph is through its maximum average degree. The \emph{average degree} $\ad$ of a graph $G=(V,E)$ is defined by $\ad(G)=\frac{2|E|}{|V|}$. The \emph{maximum average degree} $\mad(G)$ is the maximum, over all subgraphs $H$ of $G$, of $\ad(H)$. Another way to measure the sparsity is through the girth, i.e. the length of a shortest cycle. We denote $g(G)$ the girth of $G$. Intuitively, the higher the girth of a graph is, the sparser it gets. These two measures can actually be linked directly in the case of planar graphs.

A graph is \emph{planar} if one can draw its vertices with points on the plane, and edges with curves intersecting only at its endpoints. When $G$ is a planar graph, Wegner conjectured in 1977 that  $\chi^2(G)$ becomes linear in $\Delta(G)$:

\begin{conjecture}[Wegner \cite{wegner}]
\label{conj:Wegner}
Let $G$ be a planar graph with maximum degree $\Delta$. Then,
$$
\chi^2(G) \leq \left\{
    \begin{array}{ll}
        7, & \mbox{if } \Delta\leq 3, \\
        \Delta + 5, & \mbox{if } 4\leq \Delta\leq 7,\\
        \left\lfloor\frac{3\Delta}{2}\right\rfloor + 1, & \mbox{if } \Delta\geq 8.
    \end{array}
\right.
$$
\end{conjecture}

The upper bound for the case where $\Delta\geq 8$ is tight (see \Cref{wegner figure}(i)). Recently, the case $\Delta\leq 3$ was proved by Thomassen \cite{tho18}, and by Hartke \textit{et al.} \cite{har16} independently. For $\Delta\geq 8$, Havet \textit{et al.} \cite{havet} proved that the bound is $\frac{3}{2}\Delta(1+o(1))$, where $o(1)$ is as $\Delta\rightarrow\infty$ (this bound holds for 2-distance list-colorings). \Cref{conj:Wegner} is known to be true for some subfamilies of planar graphs, for example $K_4$-minor free graphs \cite{lwz03}.

\begin{figure}[htbp]
\begin{subfigure}[b]{0.48\textwidth}
\centering
\begin{tikzpicture}[scale=0.4]
\begin{scope}[every node/.style={circle,thick,draw,minimum size=1pt,inner sep=2}]
    \node[fill] (y) at (0,0) {};
    \node[fill] (z) at (5,0) {};
    \node[fill] (x) at (2.5,4.33) {};

    \node[fill] (xy) at (1.25,2.165) {};
    \node[fill] (yz) at (2.5,0) {};
    \node[fill] (zx) at (3.75,2.165) {};

    \node[fill,label={[label distance=-1cm]above:$\lfloor\frac{\Delta}{2}\rfloor-1$ vertices}] (xy1) at (-2.5,4.33) {};
    \node[fill] (xy2) at (-1.5625,3.78875) {};
    \node[fill] (xy3) at (-0.625,3.2475) {};

    \node[fill,label={[label distance=-0.7cm]above:$\lceil\frac{\Delta}{2}\rceil$ vertices}] (zx1) at (7.5,4.33) {};
    \node[fill] (zx2) at (6.5625,3.78875) {};
    \node[fill] (zx3) at (5.625,3.2475) {};

    \node[fill,label={[label distance = -0.7cm]below:$\lfloor\frac{\Delta}{2}\rfloor$ vertices}] (yz1) at (2.5,-4.33) {};
    \node[fill] (yz2) at (2.5,-3.2475) {};
    \node[fill] (yz3) at (2.5,-2.165) {};

\end{scope}

\begin{scope}[every edge/.style={draw=black,thick}]
    \path (x) edge (y);
    \path (y) edge (z);
    \path (z) edge (x);

    \path (x) edge[bend left] (y);

    \path (x) edge (xy1);
    \path (x) edge (xy2);
    \path (x) edge (xy3);

    \path (y) edge (xy1);
    \path (y) edge (xy2);
    \path (y) edge (xy3);

    \path (x) edge (zx1);
    \path (x) edge (zx2);
    \path (x) edge (zx3);

    \path (z) edge (zx1);
    \path (z) edge (zx2);
    \path (z) edge (zx3);

    \path (y) edge (yz1);
    \path (y) edge (yz2);
    \path (y) edge (yz3);

    \path (z) edge (yz1);
    \path (z) edge (yz2);
    \path (z) edge (yz3);

    \path[dashed] (xy) edge (xy3);
    \path[dashed] (yz) edge (yz3);
    \path[dashed] (zx) edge (zx3);
\end{scope}
\draw[rotate=-30] (-0.625-1.4,3.2475-0.7) ellipse (3cm and 0.5cm);
\draw[rotate=30] (5+0.625+1,3.2475-3.25) ellipse (3cm and 0.5cm);
\draw[rotate=90] (-2,-2.5) ellipse (3cm and 0.5cm);
\end{tikzpicture}
\vspace{-0.9cm}
\caption{A graph with girth 3 and $\chi^2=\lfloor\frac{3\Delta}{2}\rfloor+1$.}
\end{subfigure}
\begin{subfigure}[b]{0.48\textwidth}
\centering
\begin{tikzpicture}[scale=0.4]
\begin{scope}[every node/.style={circle,thick,draw,minimum size=1pt,inner sep=2}]
    \node[fill] (y) at (0,0) {};
    \node[fill] (z) at (5,0) {};
    \node[fill] (x) at (2.5,4.33) {};

    \node[fill] (xy) at (1.25,2.165) {};
    \node[fill] (yz) at (2.5,0) {};
    \node[fill] (zx) at (3.75,2.165) {};

    \node[fill,label={[label distance=-1cm]above:$\lfloor\frac{\Delta}{2}\rfloor-1$ vertices}] (xy1) at (-2.5,4.33) {};
    \node[fill] (xy2) at (-1.5625,3.78875) {};
    \node[fill] (xy3) at (-0.625,3.2475) {};

    \node[fill,label={[label distance=-0.7cm]above:$\lceil\frac{\Delta}{2}\rceil$ vertices}] (zx1) at (7.5,4.33) {};
    \node[fill] (zx2) at (6.5625,3.78875) {};
    \node[fill] (zx3) at (5.625,3.2475) {};

    \node[fill,label={[label distance = -0.7cm]below:$\lfloor\frac{\Delta}{2}\rfloor$ vertices}] (yz1) at (2.5,-4.33) {};
    \node[fill] (yz2) at (2.5,-3.2475) {};
    \node[fill] (yz3) at (2.5,-2.165) {};

\end{scope}

\begin{scope}[every edge/.style={draw=black,thick}]
    \path (x) edge (y);
    \path (y) edge (z);
    \path (z) edge (x);

    \path (x) edge (xy1);
    \path (x) edge (xy2);
    \path (x) edge (xy3);

    \path (y) edge (xy1);
    \path (y) edge (xy2);
    \path (y) edge (xy3);

    \path (x) edge (zx1);
    \path (x) edge (zx2);
    \path (x) edge (zx3);

    \path (z) edge (zx1);
    \path (z) edge (zx2);
    \path (z) edge (zx3);

    \path (y) edge (yz1);
    \path (y) edge (yz2);
    \path (y) edge (yz3);

    \path (z) edge (yz1);
    \path (z) edge (yz2);
    \path (z) edge (yz3);

    \path[dashed] (xy) edge (xy3);
    \path[dashed] (yz) edge (yz3);
    \path[dashed] (zx) edge (zx3);
\end{scope}
\draw[rotate=-30] (-0.625-1.4,3.2475-0.7) ellipse (3cm and 0.5cm);
\draw[rotate=30] (5+0.625+1,3.2475-3.25) ellipse (3cm and 0.5cm);
\draw[rotate=90] (-2,-2.5) ellipse (3cm and 0.5cm);
\end{tikzpicture}
\vspace{-0.9cm}
\caption{A graph with girth 4 and $\chi^2=\lfloor\frac{3\Delta}{2}\rfloor-1$.}
\end{subfigure}
\caption{Graphs with $\chi^2\approx \frac32 \Delta$.}
\label{wegner figure}
\end{figure}
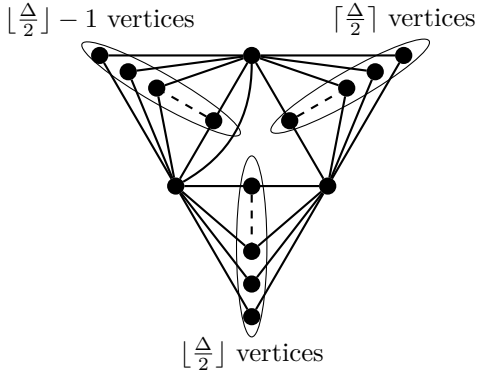
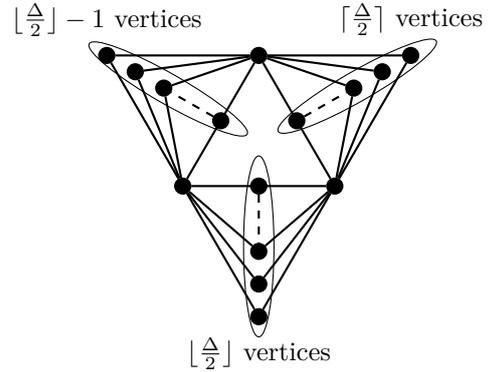

Wegner's conjecture motivated extensive researches on $2$-distance chromatic number of sparse graphs, either of planar graphs with high girth or of graphs with upper bounded maximum average degree which are directly linked due to \Cref{maximum average degree and girth proposition}.

\begin{proposition}[Folklore]\label{maximum average degree and girth proposition}
For every planar graph $G$, $(\mad(G)-2)(g(G)-2)<4$.
\end{proposition}

As a consequence, any theorem with an upper bound on $\mad(G)$ can be translated to a theorem with a lower bound on $g(G)$ under the condition that $G$ is planar. Many results have taken the following form: \textit{every graph $G$ of $\mad(G)< m_0$ and $\Delta(G)\geq \Delta_0$ satisfies $\chi^2(G)\leq \Delta(G)+c(m_0,\Delta_0)$ where $c(m_0,\Delta_0)$ is a constant depending only on $m_0$ and $\Delta_0$}. Due to \Cref{maximum average degree and girth proposition}, as a corollary, the same results on planar graphs of girth $g\geq g_0(m_0)$ where $g_0$ depends on $m_0$ follow. \Cref{recap table 2-distance} shows all known such results, up to our knowledge, on the $2$-distance chromatic number of planar graphs with fixed girth, either proven directly for planar graphs with high girth or came as a corollary of a result on graphs with bounded maximum average degree.

\begin{table}[H]
\begin{center}
\scalebox{0.7}{%
\begin{tabular}{||c||c|c|c|c|c|c|c|c||}
\hline
\backslashbox{$g_0$ \kern-1em}{\kern-1em $\chi^2(G)$} & $\Delta+1$ & $\Delta+2$ & $\Delta+3$ & $\Delta+4$ & $\Delta+5$ & $\Delta+6$ & $\Delta+7$ & $\Delta+8$\\
\hline \hline
$3$ & \slashbox{\phantom{\ \ \ \ \ }}{} & & &$\Delta=3$ \cite{tho18,har16}& & & & \\
\hline
$4$ & \slashbox{\phantom{\ \ \ \ \ }}{} & & & & & & & \\
\hline
$5$ & \slashbox{\phantom{\ \ \ \ \ }}{} &$\Delta\geq 10^7$ \cite{bon19}\footref{list footnote} &$\Delta\geq 339$ \cite{don17b} &$\Delta\geq 312$ \cite{don17} &$\Delta\geq 15$ \cite{bu18b}\tablefootnote{\label{other footnote}Corollaries of more general colorings of planar graphs.} &$\Delta\geq 12$ \cite{bu16}\footref{list footnote} & $\Delta\neq 7,8$ \cite{don17} &all $\Delta$ \cite{dl16}\\
\hline
$6$ & \slashbox{\phantom{\ \ \ \ \ }}{} &$\Delta\geq 17$ \cite{bon14}\footref{mad footnote} &$\Delta\geq 9$ \cite{bu16}\footref{list footnote} & &all $\Delta$ \cite{bu11} & & & \\
\hline
$7$ & $\Delta\geq 16$ \cite{iva11}\tablefootnote{\label{list footnote}Corollaries of 2-distance list-colorings of planar graphs.} &\cellcolor{gray!50} $\Delta \geq 10$\tablefootnote{\label{ours} Our results.} & $\Delta\geq 6$ \cite{la21}\footref{mad list footnote} &$\Delta=4$ \cite{cra13}\tablefootnote{\label{mad list footnote}Corollaries of 2-distance list-colorings of graphs with a bounded maximum average degree.} & & & & \\
\hline
$8$ & $\Delta\geq 9$ \cite{lmpv19}\footref{other footnote}& \cellcolor{gray!50}$\Delta \geq 6$\footref{ours} & $\Delta\geq 4$ \cite{la21}\footref{mad list footnote} & & & & & \\
\hline
$9$ &$\Delta\geq 7$ \cite{lm21}\footref{mad footnote} &$\Delta=5$ \cite{bu15}\footref{mad list footnote} &$\Delta=3$ \cite{cra07}\footref{list footnote} & & & & & \\
\hline
$10$ & $\Delta\geq 6$ \cite{iva11}\footref{list footnote} & & & & & & & \\
\hline
 $11$ & &$\Delta=4$ \cite{cra13}\footref{mad list footnote} & & & & & & \\
\hline
$12$ & $\Delta=5$ \cite{iva11}\footref{list footnote} &$\Delta=3$ \cite{bi12}\footref{list footnote} & & & & & &\\
\hline
$13$ & & & & & & & &\\
\hline
$14$ &$\Delta\geq 4$ \cite{bon13}\tablefootnote{\label{mad footnote}Corollaries of 2-distance colorings of graphs with a bounded maximum average degree.} & & & & & & &\\
\hline
$\dots$ & & & & & & & & \\
\hline
$21$ & $\Delta=3$\cite{lm21g21d3} & & & & & & & \\
\hline
\end{tabular}}
\caption{The latest results with a coefficient 1 before $\Delta$ in the upper bound of $\chi^2$.}
\label{recap table 2-distance}
\end{center}
\end{table} 

For example, the result from line ``7'' and column ``$\Delta + 1$'' from \Cref{recap table 2-distance} reads as follows : ``\emph{every planar graph $G$ of girth at least 7  and of $\Delta$ at least 16 satisfies $\chi^2(G)\leq \Delta+1$}''. The crossed out cases in the first column correspond to the fact that, for $g_0\leq 6$, there are planar graphs $G$ with $\chi^2(G)=\Delta+2$ for arbitrarily large $\Delta$ \cite{bor04,dvo08b}. The lack of results for $g = 4$ is due to the fact that the graph in \Cref{wegner figure}(ii) has girth 4, and $\chi^2=\lfloor\frac{3\Delta}{2}\rfloor-1$ for all $\Delta$.

We are interested in the case $\chi^2(G)\leq \Delta+2$. In particular, we were looking for the smallest integer $\Delta_0$ such that every graph with maximum degree $\Delta\geq \Delta_0$ and $\mad<\frac{8}{3}$ (resp. $\mad< \frac{14}{5}$) can be $2$-distance colored with $\Delta+2$ colors. That family contains planar graphs with $\Delta\geq\Delta_0$ and girth at least $8$ (resp. $7$).

Our main results are the following:
\begin{theorem} \label{main theorem1}
If $G$ is a graph with $\mad(G)\leq \frac83$, then $G$ is $2$-distance $(\Delta(G)+2)$-colorable for $\Delta(G)\geq 6$.
\end{theorem}

\begin{theorem} \label{main theorem2}
If $G$ is a graph with $\mad(G)\leq \frac{14}5$, then $G$ is $2$-distance $(\Delta(G)+2)$-colorable for $\Delta(G)\geq 10$.
\end{theorem}

For planar graphs, we obtain the following corollaries:
\begin{corollary} 
If $G$ is a graph with $g(G)\geq 8$, then $G$ is $2$-distance $(\Delta(G)+2)$-colorable for $\Delta(G)\geq 6$.
\end{corollary}

\begin{corollary}
If $G$ is a graph with $g(G)\geq 7$, then $G$ is $2$-distance $(\Delta(G)+2)$-colorable for $\Delta(G)\geq 10$.
\end{corollary}

We will prove \Cref{main theorem1,main theorem2} respectively in \Cref{sec2,sec3} using the same scheme.

\section{Proof of \Cref{main theorem1}}
\label{sec2}

\paragraph{Notations and drawing conventions.} For $v\in V(G)$, the \emph{2-distance neighborhood} of $v$, denoted $N^*_G(v)$, is the set of 2-distance neighbors of $v$, which are vertices at distance at most two from $v$, not including $v$. We also denote $d^*_G(v)=|N^*_G(v)|$. We will drop the subscript and the argument when it is clear from the context. Also for conciseness, from now on, when we say ``to color'' a vertex, it means to color such vertex differently from all of its colored neighbors at distance at most two. Similarly, any considered coloring will be a 2-distance coloring. We say that a vertex $u$ ``sees'' a vertex $v$ if $v\in N^*_G(u)$. We also say that $u$ ``sees a color'' $c$ if there exists $v\in N^*_G(u)$ such that $v$ is colored $c$.

Some more notations:
\begin{itemize}
\item A \emph{$d$-vertex} ($d^+$-vertex, $d^-$-vertex) is a vertex of degree $d$ (at least $d$, at most $d$). A \emph{$(\db{d}{e})$-vertex} is a vertex of degree between $d$ and $e$ included. 
\item A \emph{$k$-path} ($k^+$-path, $k^-$-path) is a path of length $k+1$ (at least $k+1$, at most $k+1$) where the $k$ internal vertices are 2-vertices. The endvertices of a $k$-path are $3^+$-vertices.
\item A \emph{$(k_1,k_2,\dots,k_d)$-vertex} is a $d$-vertex incident to $d$ different paths, where the $i^{\rm th}$ path is a $k_i$-path for all $1\leq i\leq d$.
\end{itemize}
As a drawing convention for the rest of the figures, black vertices will have a fixed degree, which is represented, and white vertices may have a higher degree than what is drawn. Also, we will represented the lower bound on the number of available colors next to each not yet colored vertex in a subgraph $H$ of $G$ when $G-H$ is colored.

Let $G_1$ be a counterexample to \Cref{main theorem1} with the fewest number of vertices. Graph $G_1$ has maximum degree $\Delta\geq 6$ and $\mad(G)<\frac83$. The purpose of the proof is to prove that $G_1$ cannot exist. In the following we will study the structural properties of $G_1$. We will then apply a discharging procedure. 

\subsection{Structural properties of $G_1$}

\begin{lemma}\label{connected1}
Graph $G_1$ is connected.
\end{lemma}

\begin{proof}
Otherwise a component of $G_1$ would be a smaller counterexample.
\end{proof}

\begin{lemma}\label{minimumDegree1}
The minimum degree of $G_1$ is at least $2$.
\end{lemma}

\begin{proof}
By \Cref{connected1}, the minimum degree is at least 1. If $G_1$ contains a degree 1 vertex $v$, then we can simply remove $v$ and 2-distance color the resulting graph, which is possible by minimality of $G_1$. Then, we add $v$ back and color it (at most $\Delta$ constraints and $\Delta+2$ colors).
\end{proof}

\begin{lemma} \label{3-path lemma1}
Graph $G_1$ has no $3^+$-paths.
\end{lemma}

\begin{proof}
Suppose $G_1$ contains a $3^+$-path $v_0v_1v_2v_3\dots v_k$ with $k\geq 4$. We color $H=G_1-\{v_1,v_2,v_3\}$ by minimality of $G_1$, then we finish by coloring $v_1$, $v_3$, and $v_2$ in this order, which is possible since they have at least respectively 2, 2, and $\Delta\geq 6$ available colors left after the coloring of $H$.
\end{proof}

\begin{lemma}\label{2-path lemma1}
A $2$-path has two distinct endvertices and both have degree $\Delta$.
\end{lemma}

\begin{proof}
Suppose that $G_1$ contains a 2-path $v_0v_1v_2v_3$. 

If $v_0=v_3$, then we color $G_1-\{v_1,v_2\}$ by minimality of $G_1$ and extend the coloring to $G_1$ by coloring greedily $v_1$ and $v_2$ who has 3 available colors each.

Now, suppose that $v_0\neq v_3$, and that $d(v_3)\leq \Delta-1$. We color $G_1-\{v_1,v_2\}$ by minimality of $G_1$ and extend the coloring to $G_1$ by coloring $v_1$ then $v_2$, which is possible since they have respectively 1 and 2 available colors left. Thus, $d(v_3)=\Delta$ and the same holds for $d(v_0)$ by symmetry.
\end{proof}

\begin{figure}[H]
\begin{subfigure}[b]{0.32\textwidth}
\centering
\begin{tikzpicture}[scale=0.6]{thick}
\begin{scope}[every node/.style={circle,draw,minimum size=1pt,inner sep=2}]
	\node[label={above:$v_0$}] (0) at (-2,0) {};
    \node[fill,label={above:$v_1$},label={below:$2$}] (1) at (0,0) {};
    \node[fill,label={above:$v_2$},label={below:$6$}] (2) at (2,0) {};
    \node[fill,label={above:$v_3$},label={below:$2$}] (3) at (4,0) {};
    \node[label={above:$v_4$}] (4) at (6,0) {};
\end{scope}

\begin{scope}[every edge/.style={draw=black}]
    \path (0) edge (4);
\end{scope}
\end{tikzpicture}
\caption{A $3^+$-path.}
\end{subfigure}
\begin{subfigure}[b]{0.32\textwidth}
\centering
\begin{tikzpicture}[scale=0.6]{thick}
\begin{scope}[every node/.style={circle,draw,minimum size=1pt,inner sep=2}]
	\node[label={left:$v_0$}] (0) at (0,0) {};
    \node[fill,label={above:$v_1$},label={left:$3$}] (1) at (-1,1) {};
    \node[fill,label={above:$v_2$},label={right:$3$}] (2) at (1,1) {};
\end{scope}

\begin{scope}[every edge/.style={draw=black}]
    \path (0) edge (1);
    \path (1) edge (2);
    \path (2) edge (0);
\end{scope}
\end{tikzpicture}
\caption{A $2$-path where both endvertices are the same.}
\end{subfigure}
\begin{subfigure}[b]{0.32\textwidth}
\centering
\begin{tikzpicture}[scale=0.6]{thick}
\begin{scope}[every node/.style={circle,draw,minimum size=1pt,inner sep=2}]
	\node[label={above:$v_0$}] (0) at (-2,0) {};
    \node[fill,label={above:$v_1$},label={below:$1$}] (1) at (0,0) {};
    \node[fill,label={above:$v_2$},label={below:$2$}] (2) at (2,0) {};
    \node[ellipse,label={above:$v_3$}] (3) at (5.5,0) {$(\Delta-1)^-$};
\end{scope}

\begin{scope}[every edge/.style={draw=black}]
    \path (0) edge (3);
\end{scope}
\end{tikzpicture}
\caption{A $2$-path incident to a $(\Delta-1)^-$-vertex.}
\end{subfigure}
\caption{\ }
\end{figure}

\begin{lemma} \label{tree lemma1}
Graph $G_1$ has no cycles consisting of $2$-paths.
\end{lemma}

\begin{proof}
Suppose that $G_1$ contains a cycle consisting of $k$ 2-paths (see \Cref{tree lemma figure}). We remove all vertices $v_{3i+1}$ and $v_{3i+2}$ for $0\leq i\leq k-1$. Consider a coloring of the resulting graph. It is then possible to color $v_1, v_2, v_4, \dots, v_{3k-1}$ since each of them has at least two choices of colors (as $d(v_0)=d(v_3)=\dots=d(v_{3(k-1)})=\Delta$ due to \Cref{2-path lemma1}) and by 2-choosability of even cycles.
\end{proof}
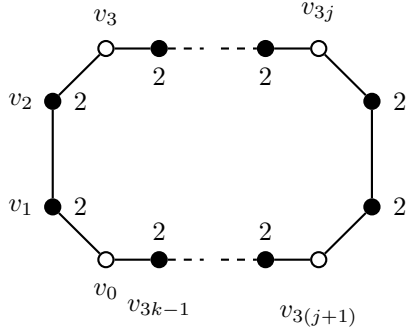
\begin{figure}[H]
\begin{center}
\begin{tikzpicture}[scale=0.7,rotate=90]
\begin{scope}[every node/.style={circle,thick,draw,minimum size=1pt,inner sep=2}]
    \node[label={above:$v_3$}] (0) at (-1,-1) {};

    \node[label={below:$v_0$}] (100) at (-5,-1) {};

    \node[fill,label={left:$v_2$},label={right:$2$}] (1) at (-2,0) {};
    \node[fill,label={left:$v_1$},label={right:$2$}] (3) at (-4,0) {};

    \node[fill,label={below:$2$}] (4) at (-1,-2) {};
    \node[draw=none] (5) at (-1,-3) {};
    \node[fill,label={below:$2$}] (6) at (-1,-4) {};

    \node[label={above:$v_{3j}$}] (200) at (-1,-5) {};

    \node[fill,label={right:$2$}] (7) at (-2,-6) {};
    \node[fill,label={right:$2$}] (9) at (-4,-6) {};

    \node[label={below:$v_{3(j+1)}$}] (300) at (-5,-5) {};

    \node[fill,label={below:$v_{3k-1}$},label={above:$2$}] (14) at (-5,-2) {};
    \node[draw=none] (15) at (-5,-3) {};
    \node[fill,label={above:$2$}] (16) at (-5,-4) {};
\end{scope}

\begin{scope}[every edge/.style={draw=black,thick}]
    \path (0) edge (1);
    \path (1) edge (3);
    \path (3) edge (100);
        \path (0) edge (4);
      \path[dashed] (4) edge (5);
      \path[dashed] (5) edge (6);
      \path (6) edge (200);
          \path (200) edge (7);
      \path (7) edge (9);
      \path (9) edge (300);
          \path (100) edge (14);
      \path[dashed] (14) edge (15);
      \path[dashed] (15) edge (16);
      \path (16) edge (300);
\end{scope}
\end{tikzpicture}
\vspace{-0.5cm}
\caption{A cycle consisting of consecutive $2$-paths.}
\label{tree lemma figure}
\end{center}
\end{figure}

\begin{lemma}\label{111-5}
Consider a $(1,1,1)$-vertex $u$. The other endvertices of the $1$-paths incident to $u$ are all distincts and are $\Delta$-vertices.
\end{lemma}

\begin{proof}
Suppose there exists a $(1,1,1)$-vertex $u$ with three $2$-neighbors $u_1$, $u_2$, and $u_3$. Let $v_i$ be the other endvertex of $uu_iv_i$ for $1\leq i\leq 3$.

First, suppose by contradiction that $v_1=v_2$ (and possibly $=v_3$). We color $G_1-\{u,u_1,u_2,u_3\}$ by minimality of $G_1$. Then, we color $u_3$, $u_1$, $u_2$, and $u$ in this order, which is possible since they have at least respectively $2$, $3$, $3$, and $\Delta\geq 6$ colors. So, $v_1$, $v_2$, and $v_3$ are all distinct.

Now, suppose w.l.o.g. that $d(v_1)\leq \Delta-1$ by contradiction. We color $G_1-\{u,u_1,u_2,u_3\}$ by minimality of $G_1$. Then, we color $u_3$, $u_2$, $u_1$, and $u$ in this order. So, $d(v_1)=d(v_2)=d(v_3)=\Delta$. 
\end{proof}

\begin{figure}[H]
\begin{subfigure}[b]{0.49\textwidth}
\centering
\begin{tikzpicture}[scale=0.6]{thick}
\begin{scope}[every node/.style={circle,draw,minimum size=1pt,inner sep=2}]
	\node[label={[label distance=-10pt]below:$v_1=v_2$}] (0) at (-2,0) {};
    \node[fill,label={above:$u_1$},label={below:$3$}] (1) at (0,0) {};
    \node[fill,label={above:$u$},label={below:$6$}] (2) at (2,0) {};
    \node[fill,label={above:$u_3$},label={below:$2$}] (3) at (4,0) {};
    \node[label={above:$v_3$}] (4) at (6,0) {};
    
    \node[fill,label={above:$u_2$},label={below:$3$}] (2') at (0,2) {};
\end{scope}

\begin{scope}[every edge/.style={draw=black}]
    \path (0) edge (4);
    \path (0) edge (2');
    \path (2') edge (2);
\end{scope}
\end{tikzpicture}
\vspace*{-0.5cm}
\caption{A $(1,1,1)$-vertex that sees only two vertices at distance $2$.}
\end{subfigure}
\begin{subfigure}[b]{0.49\textwidth}
\centering
\begin{tikzpicture}[scale=0.6]{thick}
\begin{scope}[every node/.style={circle,draw,minimum size=1pt,inner sep=2}]
	\node[ellipse,label={above:$v_1$}] (0) at (-3.5,0) {$(\Delta-1)^-$};
    \node[fill,label={above:$u_1$},label={below:$3$}] (1) at (0,0) {};
    \node[fill,label={[label distance=+4pt]above left:$u$},label={below:$5$}] (2) at (2,0) {};
    \node[fill,label={above:$u_3$},label={below:$2$}] (3) at (4,0) {};
    \node[label={above:$v_3$}] (4) at (6,0) {};
    
    \node[fill,label={left:$u_2$},label={right:$2$}] (2') at (2,2) {};
    \node[label={above:$v_2$}] (2'') at (2,4) {};
\end{scope}

\begin{scope}[every edge/.style={draw=black}]
    \path (0) edge (4);
    \path (2) edge (2'');
\end{scope}
\end{tikzpicture}
\caption{A $(1,1,1)$-vertex that sees a $(\Delta-1)^-$-vertex at distance $2$.}
\end{subfigure}
\caption{\ }
\end{figure}

\begin{definition}[$(1,1,1)$-paths]
We call $v_0v_1v_2v_3v_4$ a $(1,1,1)$-path when $v_0$ and $v_4$ are $\Delta$-vertices, $v_1$ and $v_3$ are $2$-vertices, and $v_2$ is a $(1,1,1)$-vertex.
\end{definition}

\begin{lemma}\label{tree lemma1bis}
Graph $G_1$ has no cycles consisting of $(1,1,1)$-paths.
\end{lemma}

\begin{proof}
Suppose that $G$ contains a cycle consisting of $k$ $(1,1,1)$-paths (see \Cref{tree lemma figure1bis}). We remove all vertices $v_{4i+1}$, $v_{4i+2}$, $v_{4i+3}$ for $0\leq i\leq k-1$. Consider a coloring of the resulting graph. We color $v_1, v_3, v_5, \dots, v_{4k-1}$ since each of them has at least two choices of colors (as $d(v_0)=d(v_4)=\dots=d(v_{4(k-1)})=\Delta$ due to \Cref{111-5}) and by 2-choosability of even cycles. Finally, it is easy to color greedily $v_2,v_6,\dots,v_{4k-2}$ since they each have at most six forbidden colors.
\end{proof}

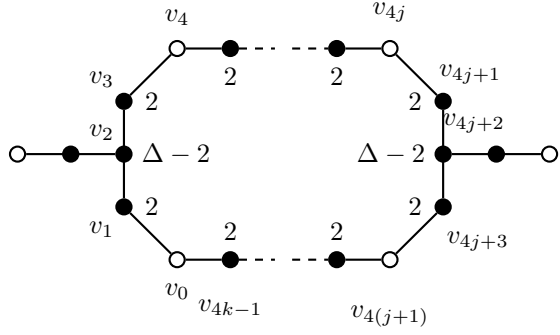
\begin{figure}[H]
\begin{center}
\begin{tikzpicture}[scale=0.7,rotate=90]
\begin{scope}[every node/.style={circle,thick,draw,minimum size=1pt,inner sep=2}]
    \node[label={above:$v_4$}] (0) at (-1,-1) {};

    \node[label={below:$v_0$}] (100) at (-5,-1) {};

    \node[fill,label={above left:$v_3$},label={right:$2$}] (1) at (-2,0) {};
    \node[fill,label={above left:$v_2$},label={right:$\Delta-2$}] (2) at (-3,0) {};
    \node[fill] (2') at (-3,1) {};
    \node (2'') at (-3,2) {};
    \node[fill,label={below left:$v_1$},label={right:$2$}] (3) at (-4,0) {};

    \node[fill,label={below:$2$}] (4) at (-1,-2) {};
    \node[draw=none] (5) at (-1,-3) {};
    \node[fill,label={below:$2$}] (6) at (-1,-4) {};

    \node[label={above:$v_{4j}$}] (200) at (-1,-5) {};

    \node[fill,label={[label distance=-4pt]above right:$v_{4j+1}$},label={left:$2$}] (7) at (-2,-6) {};
    \node[fill,label={[label distance=-2pt]above right:$v_{4j+2}$},label={left:$\Delta-2$}] (8) at (-3,-6) {};
    \node[fill] (8') at (-3,-7) {};
    \node (8'') at (-3,-8) {};
    \node[fill,label={below right:$v_{4j+3}$},label={left:$2$}] (9) at (-4,-6) {};

    \node[label={below:$v_{4(j+1)}$}] (300) at (-5,-5) {};

    \node[fill,label={below:$v_{4k-1}$},label={above:$2$}] (14) at (-5,-2) {};
    \node[draw=none] (15) at (-5,-3) {};
    \node[fill,label={above:$2$}] (16) at (-5,-4) {};
\end{scope}

\begin{scope}[every edge/.style={draw=black,thick}]
    \path (0) edge (1);
    \path (1) edge (3);
    \path (3) edge (100);
        \path (0) edge (4);
      \path[dashed] (4) edge (5);
      \path[dashed] (5) edge (6);
      \path (6) edge (200);
          \path (200) edge (7);
      \path (7) edge (9);
      \path (9) edge (300);
          \path (100) edge (14);
      \path[dashed] (14) edge (15);
      \path[dashed] (15) edge (16);
      \path (16) edge (300);
     
     \path (2) edge (2'');
     \path (8) edge (8'');
\end{scope}
\end{tikzpicture}
\vspace{-0.5cm}
\caption{A cycle consisting of consecutive $(1,1,1)$-paths.}
\label{tree lemma figure1bis}
\end{center}
\end{figure}

\begin{lemma}\label{3-011-5}
A $(1,1,0)$-vertex with a $(\db{3}{\Delta-3})$-neighbor shares its $2$-neighbors with $\Delta$-vertices.
\end{lemma}

\begin{proof}
Suppose that there exists a $(1,1,0)$-vertex $u$ with a $(\db{3}{\Delta-3})$-neighbor. Let $u_1$ and $u_2$ be its $2$-neighbors. Let $v\neq u$ be the other neighbor of $u_1$. Suppose w.l.o.g. that $d(v)\leq \Delta-1$ by contradiction. We color $G_1-\{u,u_1,u_2\}$ by minimality of $G_1$. Then, we color $u_2$, $u_1$, and $u$ in this order, which is possible since they have at least respectively 1, 2, and 3 colors as we have $\Delta+2$ colors.
\end{proof}

\begin{figure}[H]
\centering
\begin{tikzpicture}[scale=0.6]{thick}
\begin{scope}[every node/.style={circle,draw,minimum size=1pt,inner sep=2}]
	\node[ellipse,label={above:$v$}] (0) at (-3.5,0) {$(\Delta-1)^-$};
    \node[fill,label={above:$u_1$},label={below:$2$}] (1) at (0,0) {};
    \node[fill,label={[label distance=+4pt]above left:$u$},label={below:$3$}] (2) at (2,0) {};
    \node[fill,label={above:$u_2$},label={below:$1$}] (3) at (4,0) {};
    \node (4) at (6,0) {};
    
    \node[ellipse] (2') at (2,2.5) {$\db{3}{\Delta-3}$};
\end{scope}

\begin{scope}[every edge/.style={draw=black}]
    \path (0) edge (4);
    \path (2) edge (2');
\end{scope}
\end{tikzpicture}
\caption{A $(1,1,0)$-vertex with a $(\db{3}{\Delta-3})$-neighbor that shares a $2$-neighbor with a $(\Delta-1)^-$-vertex.}
\end{figure}

\begin{lemma}\label{weird cases lemma}
A $\Delta$-vertex $u$ cannot be incident to a $2$-path, a $(1,1,1)$-path, and $\Delta-2$ other $1^+$-paths $uu_iv_i$ ($1\leq i\leq \Delta-2$) where each $v_i$ is a $3^-$-vertex.
\end{lemma}

\begin{proof}
Let $uu_{\Delta-1}v_{\Delta-1}\notin \{uu_iv_i|1\leq i\leq \Delta-2\}$ be a $1$-path where $v_{\Delta-1}$ is a $(1,1,1)$-vertex. Let $uu_{\Delta}u'_{\Delta}v_{\Delta}$ be a $2$-path incident to $u$ where $uu_{\Delta}u'_{\Delta}\notin \{uu_iv_i|1\leq i\leq \Delta-1\}$. Observe that $v_{\Delta-1}\notin \{v_i|1\leq i\leq \Delta-2\}$ due to \Cref{111-5} and $v_{\Delta}\notin \{v_i|1\leq i\leq \Delta-2\}$ due to \Cref{2-path lemma1}.

Let $H=u\cup N_G(u)\cup\{u'_{\Delta}\}$. We color $G-H$ by minimality of $G$ and we uncolor $v_{\Delta-1}$. Let $L(x)$ be the list of remaining colors for a vertex $x\in H\cup\{v_{\Delta-1}\}$. Observe that $|L(u)|\geq \Delta+2-(\Delta-2)\geq 4$, $|L(u'_{\Delta})|\geq 2$ (since $d(v_{\Delta})=\Delta$ by \Cref{2-path lemma1}), $|L(v_{\Delta-1})|\geq \Delta-2\geq 4$, $|L(u_i)|\geq \Delta-1$ (since $d(v_i)\leq 3$) for $1\leq i\leq \Delta-2$, $|L(u_{\Delta-1})|\geq \Delta$, and $|L(u_{\Delta})|\geq \Delta+1$. We remove the extra colors from $L(u'_{\Delta})$ so that $|L(u'_{\Delta})|=2$. We color $u$ with a color that is not in $L(u'_{\Delta})$, then $u_1$, $u_2$, $\dots$, $u_{\Delta}$, $v_{\Delta-1}$, and $u'_{\Delta}$ in this order. Observe that when $v_i=v_j$ for $1\leq i\leq j\leq \Delta-2$, then $|L(u_i)|\geq \Delta$ and $|L(u_j)|\geq \Delta$ so the order in our coloring still hold. Thus, we obtain a valid coloring of $G$, which is a contradiction.
\end{proof}

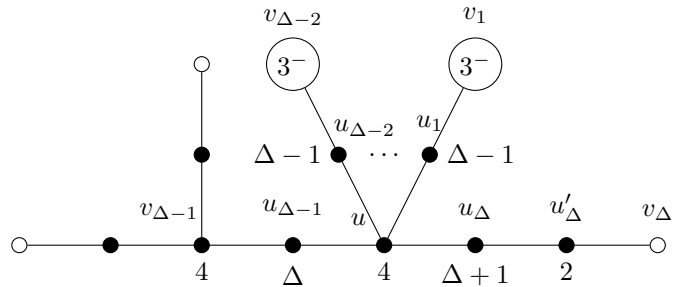
\begin{figure}[H]
\centering
\begin{tikzpicture}[scale=0.6]{thick}
\begin{scope}[every node/.style={circle,draw,minimum size=1pt,inner sep=2}]
	\node (02) at (-6,0) {};	
	\node[fill] (01) at (-4,0) {}; 	
	\node[fill,label={above left:$v_{\Delta-1}$},label={below:$4$}] (0) at (-2,0) {};
    \node[fill,label={[label distance=-4pt]above:$u_{\Delta-1}$},label={below:$\Delta$}] (1) at (0,0) {};
    \node[fill,label={[label distance=+4pt]above left:$u$},label={below:$4$}] (2) at (2,0) {};
    \node[fill,label={above:$u_{\Delta}$},label={[label distance=-8pt]below:$\Delta+1$}] (3) at (4,0) {};
    \node[fill,label={above:$u'_{\Delta}$},label={below:$2$}] (4) at (6,0) {};
	\node[label={above:$v_{\Delta}$}] (5) at (8,0) {};
    
    \node[fill] (0') at (-2,2) {};
    \node (0'') at (-2,4) {};
    
    \node[fill,label={above:$u_1$},label={right:$\Delta-1$},label={[label distance=+4pt]left:$\dots$}] (2') at (3,2) {};
    \node[label={above:$v_1$}] (2'') at (4,4) {$3^-$};
    
    \node[fill,label={[label distance=-4pt]above right:$u_{\Delta-2}$},label={left:$\Delta-1$}] (21') at (1,2) {};
    \node[label={[label distance=-7pt]above:$v_{\Delta-2}$}] (21'') at (0,4) {$3^-$};
   
\end{scope}

\begin{scope}[every edge/.style={draw=black}]
    \path (02) edge (5);
    \path (0) edge (0'');
    \path (2) edge (2'');
    \path (2) edge (21'');
\end{scope}
\end{tikzpicture}
\caption{A $\Delta$-vertex incident to a $2$-path, a $(1,1,1)$-path, and $\Delta-2$ other $1^+$-paths with $3$-endvertices.}
\label{weird cases figure}
\end{figure}

\subsection{Discharging rules}

\begin{definition}[$2$-path sponsors] \label{2-path sponsor definition}
Consider the set of $2$-paths in $G$. By \Cref{2-path lemma1}, the endvertices of every 2-paths are $\Delta$-vertices and by \Cref{tree lemma1}, the graph induced by the edges of all the $2$-paths of $G$ is a forest ${\cal F}$. For each tree of ${\cal F}$, we choose one $\Delta$-vertex as an arbitrary root. Each $2$-path is assigned a unique {\em sponsor} which is the $\Delta$-endvertex that is further away from the root. See \Cref{2-path sponsor figure}.
\end{definition}

\begin{figure}[H]
\begin{center}
\begin{tikzpicture}[scale=0.6]
\begin{scope}[every node/.style={circle,thick,draw,minimum size=1pt,inner sep=1}]
    \node (0) at (-5,0) {$\Delta$};
    \node[fill] (3) at (-2.5,0) {};
	\node[draw=none] (2) at (-3,0.5) {};     
    \node[fill] (1) at (-3.5,0) {};

    \node[label={above:root}] (10) at (-1,0) {$\Delta$};
    \node[fill] (11) at (0.5,0) {};
	\node[draw=none] (12) at (1,0.5) {};     
    \node[fill] (13) at (1.5,0) {};

    \node (20) at (3,0) {$\Delta$};

    \node[fill] (14) at (-1,-1.5) {};
	\node[draw=none] (15) at (-0.5,-2) {};     
    \node[fill] (16) at (-1,-2.5) {};

    \node (30) at (-1,-4) {$\Delta$};
    \node[fill] (31) at (0.5,-4) {};
	\node[draw=none] (32) at (1,-3.5) {};     
    \node[fill] (33) at (1.5,-4) {};

    \node (40) at (3,-4) {$\Delta$};

    \node[fill] (34) at (-2.5,-4) {};
	\node[draw=none] (35) at (-3,-3.5) {};     
    \node[fill] (36) at (-3.5,-4) {};

    \node (50) at (-5,-4) {$\Delta$};
\end{scope}

\begin{scope}[every edge/.style={draw=black,thick}]
    \path (0) edge (10);
    \path (10) edge (20);
    \path (10) edge (30);
    \path (50) edge (30);
    \path (30) edge (40);
    \path[->] (0) edge[bend left] node[above] {sponsor} (2);
    \path[->] (50) edge[bend left] (35);
    \path[->] (40) edge[bend right] (32);
    \path[->] (20) edge[bend right] (12);
    \path[->] (30) edge[bend right] (15);
    
\end{scope}

\draw (-3,0) ellipse (1cm and 0.5cm);
\draw (1,0) ellipse (1cm and 0.5cm);
\draw (-3,-4) ellipse (1cm and 0.5cm);
\draw (1,-4) ellipse (1cm and 0.5cm);
\draw[rotate=90] (-2,1) ellipse (1cm and 0.5cm);
\end{tikzpicture}
\caption{The sponsor assignment in a tree consisting of 2-paths.}
\label{2-path sponsor figure}
\end{center}
\end{figure}
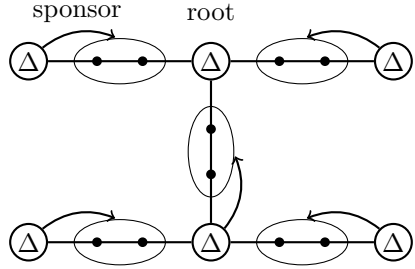

\begin{definition}[$(1,1,1)$-path sponsors] \label{111-path sponsor definition}
Consider the set of $(1,1,1)$-paths in $G$. By \Cref{111-5}, the endvertices of every $(1,1,1)$-paths are $\Delta$-vertices and by \Cref{tree lemma1bis}, the graph induced by the edges of all the $(1,1,1)$-paths of $G$ is a forest ${\cal F}$. For each tree of ${\cal F}$, we choose one $\Delta$-vertex as an arbitrary root. Each $(1,1,1)$-vertex $v$ is assigned two {\em sponsors} which are the $\Delta$-vertices that are grandsons of $v$. See \Cref{111-path sponsor figure}.
\end{definition}

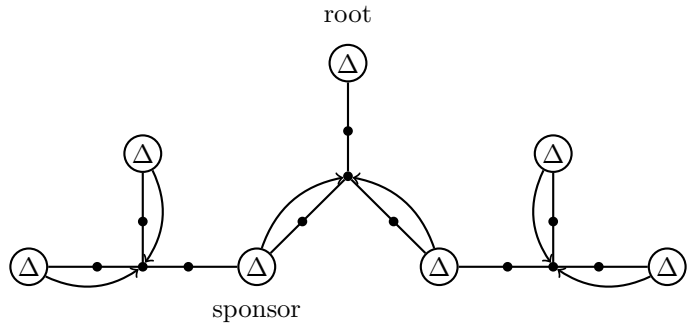
\begin{figure}[H]
\begin{center}
\begin{tikzpicture}[scale=0.6]
\begin{scope}[every node/.style={circle,thick,draw,minimum size=1pt,inner sep=1}]
    \node[label={above:root}] (10) at (-1,0.5) {$\Delta$};
    \node[fill] (14) at (-1,-1) {};
    \node[fill] (15) at (-1,-2) {};
    \node[fill] (16) at (0,-3) {};
    \node (17) at (1,-4) {$\Delta$}; 
    \node[fill] (18) at (-2,-3) {};
	\node[label={[label distance=-8pt]below:sponsor}] (19) at (-3,-4) {$\Delta$};
	
	\node[fill] (20) at (2.5,-4) {};
	\node[fill] (21) at (3.5,-4) {};
	\node[fill] (22) at (4.5,-4) {};
	\node (23) at (6,-4) {$\Delta$};
	
	\node[fill] (24) at (3.5,-3) {};
	\node (25) at (3.5,-1.5) {$\Delta$};
	
	\node[fill] (30) at (-4.5,-4) {};
	\node[fill] (31) at (-5.5,-4) {};
	\node[fill] (32) at (-6.5,-4) {};
	\node (33) at (-8,-4) {$\Delta$};
	
	\node[fill] (34) at (-5.5,-3) {};
	\node (35) at (-5.5,-1.5) {$\Delta$};
\end{scope}

\begin{scope}[every edge/.style={draw=black,thick}]
    \path (10) edge (15);
    \path (15) edge (17);
    \path (15) edge (19);
    \path (17) edge (23);
    \path (21) edge (25);
    \path (19) edge (33);
    \path (31) edge (35);
    
    \path[->] (19) edge[bend left] (15);
    \path[->] (17) edge[bend right] (15);
    \path[->] (33) edge[bend right] (31);
    \path[->] (35) edge[bend left] (31);
    \path[->] (23) edge[bend left] (21);
    \path[->] (25) edge[bend right] (21);
\end{scope}
\end{tikzpicture}
\caption{The sponsor assignment in a tree consisting of $(1,1,1)$-paths.}
\label{111-path sponsor figure}
\end{center}
\end{figure}

Since we have $\mad(G_1)<\frac83$, we must have 
\begin{equation}\label{equation1}
\sum_{v\in V(G_1)} \left(3d(v)-8\right) < 0
\end{equation}

We assign to each vertex $v$ the charge $\mu(v)=3d(v)-8$. To prove the non-existence of $G_1$, we will redistribute the charges preserving their sum and obtaining a non-negative total charge, which will contradict \Cref{equation1}.

\medskip

\begin{itemize}
\item[\ru0] (see \Cref{r0 figure}): Every $3^+$-vertex gives 1 to each 2-neighbor on an incident $1$-path.

\item[\ru1] (see \Cref{r1 figure}): Let $u$ be incident to a $2$-path $P=uu_1u_2v$.
\begin{itemize}
\item[(i)] If $u$ is not $P$'s sponsor, then $u$ gives $\frac32$ to $u_1$.
\item[(ii)] If $u$ is $P$'s sponsor, then $u$ gives 2 to $u_1$ and $\frac12$ to $u_2$.
\end{itemize}

\item[\ru2] (see \Cref{r2 figure}): Every $4^+$-vertex gives 1 to each $3$-neighbor.

\item[\ru3] (see \Cref{r3 figure}): Let $uvw$ be a $1$-path.
\begin{itemize}
\item[(i)] If $u$ is a $\Delta$-vertex, $w$ is a $(1,1,1)$-vertex, and $u$ is $w$'s sponsor, then $u$ gives $1$ to $w$.
\item[(ii)] If $u$ is a $\Delta$-vertex and $w$ is a $(1,1,0)$-vertex, then $u$ gives $\frac12$ to $w$.
\end{itemize}

\end{itemize}

\renewcommand{\thesubfigure}{\roman{subfigure}}

\begin{figure}[H]
\begin{minipage}[b]{0.2\textwidth}
\begin{center}
\begin{tikzpicture}[baseline=(0.north)]
\begin{scope}[every node/.style={circle,thick,draw,minimum size=1pt,inner sep=2,font=\small}]
    \node[minimum width=23pt] (0) at (1,2) {$3^+$};
    \node[minimum width=23pt] (100) at (-1,2) {$3^+$};
    \node[fill] (3) at (0,2) {};
\end{scope}

\begin{scope}[every edge/.style={draw=black,thick}]
    \path (0) edge (100);
      \path[->] (0) edge[bend right] node[above] {1} (3);
      \path[->] (100) edge[bend left] node[above] {1} (3);
\end{scope}
\end{tikzpicture}
\caption{\ru0.}
\label{r0 figure}
\end{center}
\end{minipage}
\begin{minipage}[b]{0.59\textwidth}
\begin{subfigure}[b]{0.49\textwidth}
\begin{center}
\begin{tikzpicture}[baseline=(3.north)]
\begin{scope}[every node/.style={circle,thick,draw,minimum size=1pt,inner sep=2,font=\small}]
    \node[minimum width=23pt,label={[label distance=-16pt]below:non-sponsor}] (0) at (-3,0) {$\Delta$};
    \node[fill] (2) at (-2,0) {};
    \node[fill] (3) at (-1,0) {};
    \node[minimum width=23pt,label={[label distance=-8pt]below:sponsor}] (4) at (0,0) {$\Delta$};
\end{scope}

\begin{scope}[every edge/.style={draw=black,thick}]
    \path (0) edge (4);    
    \path[->] (0) edge[bend left] node[above] {$\frac32$} (2);
\end{scope}
\end{tikzpicture}
\vspace*{-0.9cm}
\caption{\ }
\end{center}
\end{subfigure}
\begin{subfigure}[b]{0.49\textwidth}
\begin{center}
\begin{tikzpicture}[baseline=(3.north)]
\begin{scope}[every node/.style={circle,thick,draw,minimum size=1pt,inner sep=2,font=\small}]
    \node[minimum width=23pt,label={[label distance=-16pt]below:non-sponsor}] (0) at (-3,0) {$\Delta$};
    \node[fill] (2) at (-2,0) {};
    \node[fill] (3) at (-1,0) {};
    \node[minimum width=23pt,label={[label distance=-8pt]below:sponsor}] (4) at (0,0) {$\Delta$};
\end{scope}

\begin{scope}[every edge/.style={draw=black,thick}]
    \path (0) edge (4);    
    \path[->] (4) edge[bend left] node[below] {$2$} (3);
    \path[->] (4) edge[bend right] node[above] {$\frac12$} (2);
\end{scope}
\end{tikzpicture}
\vspace*{-0.9cm}
\caption{$2$-path sponsor.}
\end{center}
\end{subfigure}
\caption{\ru1.}
\label{r1 figure}
\end{minipage}
\begin{minipage}[b]{0.2\textwidth}
\begin{center}
\begin{tikzpicture}[baseline=(0.north)]
\begin{scope}[every node/.style={circle,thick,draw,minimum size=1pt,inner sep=2,font=\small}]
    \node[minimum width=23pt] (0) at (1,2) {$3$};
    \node[minimum width=23pt] (100) at (-0.5,2) {$4^+$};
\end{scope}

\begin{scope}[every edge/.style={draw=black,thick}]
    \path (0) edge (100);
      \path[->] (100) edge[bend left] node[above] {1} (0);
\end{scope}
\end{tikzpicture}
\caption{\ru2.}
\label{r2 figure}
\end{center}
\end{minipage}
\end{figure}

\begin{figure}[H]
\begin{center}
\begin{subfigure}[b]{0.49\textwidth}
\begin{center}
\begin{tikzpicture}[baseline=(3.north)]
\begin{scope}[every node/.style={circle,thick,draw,minimum size=1pt,inner sep=2,font=\small}]
    \node[minimum width=23pt,label={[label distance=-16pt]below:non-sponsor}] (1) at (-2,0) {$\Delta$};
    \node[fill] (1') at (-1,0) {};
    \node[fill,label={above:$w$}] (2) at (0,0) {};
    
    \node[fill,label={above:$v$}] (30) at (1,0.25) {};
    \node[minimum width=23pt,label={above:$u$}] (50) at (2,0.5) {$\Delta$};
    
    \node[fill] (31) at (1,-0.25) {};
    \node[minimum width=23pt] (51) at (2,-0.5) {$\Delta$};
    
    \node[draw=none] (100) at (3,-0.1) {sponsors};
\end{scope}

\begin{scope}[every edge/.style={draw=black,thick}]    
    \path (1) edge (2);
	
	\path (2) edge (50);
	\path (2) edge (51);
	
    \path[->] (50) edge[bend right=40] node[above] {$1$} (2);
    \path[->] (51) edge[bend left=40] node[below] {$1$} (2);
\end{scope}
\end{tikzpicture}
\vspace*{-0.8cm}
\caption{$(1,1,1)$-path sponsor.}
\end{center}
\end{subfigure}
\begin{subfigure}[b]{0.49\textwidth}
\begin{center}
\begin{tikzpicture}[baseline=(3.north)]
\begin{scope}[every node/.style={circle,thick,draw,minimum size=1pt,inner sep=2,font=\small}]
    \node[minimum width=23pt,label={above:$u$}] (1) at (-2,0) {$\Delta$};
    \node[fill,label={above:$v$}] (1') at (-1,0) {};
    \node[fill,label={above:$w$}] (2) at (0,0) {};
    
    \node[fill] (30) at (1,0.25) {};
    \node[minimum width=23pt] (50) at (2,0.5) {$3^+$};
    
    \node[minimum width=23pt] (51) at (2,-0.5) {$3^+$};
\end{scope}

\begin{scope}[every edge/.style={draw=black,thick}]    
    \path (1) edge (2);
	
	\path (2) edge (50);
	\path (2) edge (51);
	
    \path[->] (1) edge[bend left=40] node[above] {$\frac12$} (2);
\end{scope}
\end{tikzpicture}
\caption{\ }
\end{center}
\end{subfigure}
\caption{\ru3.}
\label{r3 figure}
\end{center}
\end{figure}

\subsection{Verifying that charges on each vertex are non-negative} \label{verification1}

Let $\mu^*$ be the assigned charges after the discharging procedure. In what follows, we prove that: $$\forall u \in V(G_1), \mu^*(u)\ge 0.$$

Let $u\in V(G_1)$.

\textbf{Case 1: }If $d(u)=2$, then recall that $\mu(u)=3\cdot 2 - 8 = - 2$.\\
There are no $3^+$-paths due to \Cref{3-path lemma1} so $u$ must lie on a $1$-path or a $2$-path.

If $u$ is on a $1$-path, then it has two $3^+$-neighbors which give it 1 each by \ru0. Thus,
$$ \mu^*(u)= -2 + 2\cdot 1 = 0.$$

If $u$ is on a $2$-path, then it either receives 2 from an adjacent sponsor by \ru1(ii), or it receives $\frac32+\frac12 = 2$ from an adjacent non-sponsor $\Delta$-neighbor and a distance 2 sponsor respectively by \ru1(i) and \ru1(ii). Thus,
$$ \mu^*(u)= -2 + 2 = 0.$$

\textbf{Case 2: }If $d(u)=3$, then recall that $\mu(u)=3\cdot 3 - 8 = 1$.\\
Observe that $u$ only gives charge away by \ru0 (charge 1 to each $2$-neighbor).

If $u$ is a $(1,1,1)$-vertex, then the other endvertices of the 1-paths incident to $u$ are all $\Delta$-vertices due to \Cref{111-5}. Moreover, by \Cref{111-path sponsor definition}, $u$ has two sponsors which give it $1$ each by \ru3(i). Hence,
$$ \mu^*(u)= 1 -3\cdot 1 + 2\cdot 1 = 0.$$

If $u$ is a  $(1,1,0)$-vertex with a $4^+$-neighbor, then it receives 1 from its neighbor by \ru2. Thus,
$$ \mu^*(u)= 1 -2\cdot 1 + 1 = 0.$$

If $u$ is a  $(1,1,0)$-vertex with a $3$-neighbor ($3\leq \Delta-3$ since $\Delta\geq 6$), then it receives $\frac12$ by \ru3(ii) from each of the other endvertices of its incident $1$-paths due to \Cref{3-011-5}. Thus,
$$ \mu^*(u)= 1 -2\cdot 1 + 2\cdot \frac12 = 0.$$

If $u$ is a $(1^-,0,0)$-vertex, then
$$ \mu^*(u)\geq 1 - 1 = 0.$$

\textbf{Case 3: }If $4\leq d(u)\leq \Delta-1$, then $u$ only gives away at most 1 to each neighbor by \ru0 or \ru2. Thus,
$$ \mu^*(u)\geq 3d(u) - 8 - d(u) \geq 2\cdot 4 - 8 = 0.$$ 

\textbf{Case 4: }If $d(u)=\Delta$, then we distinguish the following cases.
\begin{itemize}
\item If $u$ is neither a $2$-path sponsor nor a $(1,1,1)$-path sponsor, then observe that $u$ gives away at most $\frac32$ along an incident path by \ru1(i), a combination of \ru0 and \ru3(i), or less by \ru2. So at worst, $$\mu^*(u)\geq 3\Delta - 8 - \frac32\Delta \geq \frac32\cdot 6 - 8 =1.$$

\item If $u$ is a $2$-path sponsor but not a $(1,1,1)$-path sponsor, then $u$ gives $2+\frac12=\frac52$ to its unique sponsored incident $2$-path by \ru1(ii). For the other incident paths, it gives at most $\frac32$ like above. So,
$$ \mu^*(u)\geq 3\Delta - 8 - \frac52 - \frac32(\Delta-1) \geq \frac32\cdot 6 - 8 - \frac52 + \frac32 = 0.$$

\item If $u$ is a $(1,1,1)$-path sponsor but not a $2$-path sponsor, then $u$ gives $1+1=2$ to the unique incident $(1,1,1)$-path containing its assigned $(1,1,1)$-vertex $v$: $1$ to the $2$-neighbor by \ru0 and $1$ to $v$ by \ru3(i). Once again, $u$ gives at most $\frac32$ to the other incident paths. So,
$$ \mu^*(u)\geq 3\Delta - 8 - 2 - \frac32(\Delta-1) \geq \frac32\cdot 6 - 8 - 2 + \frac32 = \frac12.$$

\item If $u$ is both a $2$-path sponsor and a $(1,1,1)$-path sponsor, then $u$ gives $\frac52$ to its unique sponsored $2$-path and $2$ to its unique assigned $(1,1,1)$-vertex like above. 

Now, let us consider the other $\Delta-2$ paths incident to $u$. Observe that when $u$ gives $\frac32$ along an incident path either by \ru1(i) or by a combination of \ru0 and \ru3(ii), that path must be a $1^+$-path where the vertex at distance $2$ from $u$ is a $3^-$-vertex. Due to \Cref{weird cases lemma}, $u$ never has to give $\frac32$ to each of the $\Delta-2$ paths. As a result, there exists one path to which $u$ gives at most 1. So at worst,
$$ \mu^*(u)\geq 3\Delta - 8 - \frac52 -2 -1 - \frac32(\Delta-3) \geq \frac32\cdot 6 - 8 -\frac52 -2 -1 + \frac92 = 0. $$ 
\end{itemize}

We obtain a non-negative amount of charge on each vertex, which is impossible since the total amount of charge is negative. As such, $G_1$ cannot exist. That concludes the proof of \Cref{main theorem1}.

\section{Proof of \Cref{main theorem2}}
\label{sec3}

We will reuse similar notations to \Cref{sec2}. Let $G_2$ be a counterexample to \Cref{main theorem2} with the fewest number of vertices. Graph $G_2$ has maximum degree $\Delta\geq 10$ and $\mad<\frac{14}{5}$. The purpose of the proof is to prove that $G_2$ cannot exist.

\subsection{Structural properties of $G_2$}

Observe that the proofs of \Cref{connected1,minimumDegree1,3-path lemma1,2-path lemma1,tree lemma1,111-5,3-011-5,tree lemma1bis,weird cases lemma} only rely on the facts that we have a minimal counter-example, two more colors than the maximum degree, and that $\Delta$ was large enough $(\Delta(G_1)\geq 6)$. All of these still hold for $G_2$ $(\Delta(G_2)\geq 10)$. Thus, we also have the following.

\begin{lemma}\label{connected2}
Graph $G_2$ is connected.
\end{lemma}

\begin{lemma}\label{minimumDegree2}
The minimum degree of $G_2$ is at least 2.
\end{lemma}

\begin{lemma} \label{3-path lemma2}
Graph $G_2$ has no $3^+$-paths.
\end{lemma}

\begin{lemma}\label{2-path lemma2}
A $2$-path has two distinct endvertices and both have degree $\Delta$.
\end{lemma}

\begin{lemma} \label{tree lemma2}
Graph $G_2$ has no cycles consisting of $2$-paths.
\end{lemma}

\begin{lemma} \label{111-10}
Consider a $(1,1,1)$-vertex $u$. The other endvertices of the $1$-paths incident to $u$ are all distincts and are $\Delta$-vertices.
\end{lemma}

\begin{lemma} \label{tree lemma2bis}
Graph $G_2$ has no cycles consisting of $(1,1,1)$-paths.
\end{lemma}

\begin{lemma} \label{8-011-10}
A $(1,1,0)$-vertex with a $(\db{3}{\Delta-3})$-neighbor shares its $2$-neighbors with $\Delta$-vertices.
\end{lemma}

\begin{lemma}\label{weird cases lemma2}
A $\Delta$-vertex $u$ cannot be incident to a $2$-path, a $(1,1,1)$-path, and $\Delta-2$ other $1^+$-paths $uu_iv_i$ ($1\leq i\leq \Delta-2$) where each $v_i$ is a $3^-$-vertex.
\end{lemma}

We will show some more reducible configurations.

\begin{lemma} \label{4-00(-4)1-10}
A $(1,0,0)$-vertex with two $(\db34)$-neighbors shares its $2$-neighbor with a $\Delta$-vertex.
\end{lemma}

\begin{proof}
Suppose by contradiction that there exists a $(1,0,0)$-vertex $u$ with two $(\db34)$-neighbors $u_1$, $u_2$, and let $uvw$ be the $1$-path incident to $u$, where $d(w)\leq \Delta-1$. We color $G_2-\{v\}$ by minimality of $G_2$, then we uncolor $u$. Since we have $\Delta+2\geq 12$ colors and $d^*(u)=d(u_1)+d(u_2)+2\leq 4+4+2=10$, we can always color $u$ last. Finally, $v$ has at least one available color. Thus, we obtain a valid coloring of $G_2$, which is a contradiction. 
\end{proof}

\begin{figure}[H]
\centering
\begin{tikzpicture}[scale=0.6]{thick}
\begin{scope}[every node/.style={circle,draw,minimum size=1pt,inner sep=2}]
	\node[ellipse,label={above:$u_1$}] (0) at (-1.5,0) {$\db34$};
    \node[fill,label={[label distance=+4pt]above left:$u$},label={below:$2$}] (2) at (2,0) {};
    \node[fill,label={above:$v$},label={below:$1$}] (3) at (4,0) {};
    \node[ellipse,label={above:$w$}] (4) at (7.5,0) {$(\Delta-1)^-$};
    
    \node[ellipse,label={above:$u_2$}] (2') at (2,2.5) {$\db34$};
\end{scope}

\begin{scope}[every edge/.style={draw=black}]
    \path (0) edge (4);
    \path (2) edge (2');
\end{scope}
\end{tikzpicture}
\caption{A $(1,0,0)$-vertex with two $(\db34)$-neighbor that shares a $2$-neighbor with a $(\Delta-1)^-$-vertex.}
\end{figure}

\begin{lemma} \label{1111}
Consider the four other endvertices of the $1$-paths incident to a $(1,1,1,1)$-vertex. At most one of them is a $(\Delta-2)^-$-vertex.
\end{lemma}

\begin{proof}
Suppose by contradiction that we have a $(1,1,1,1)$-vertex $u$ incident to four $1$-paths $uu_iv_i$ for $1\leq i\leq 4$, where $v_1$ and $v_2$ are $(\Delta-2)^-$-vertices. We color $G_2-\{u,u_1,u_2,u_3,u_4\}$ by minimality of $G_2$. Then, it suffices to color $u_3$, $u_4$, $u_1$, $u_2$, and $u$ in this order, which is possible since they have at least respectively $2$, $2$, $4$, $4$, and $8$ available colors as we have $\Delta+2$ colors and $\Delta\geq 10$.
\end{proof}

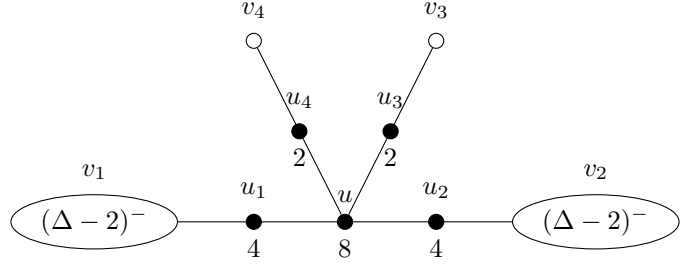
\begin{figure}[H]
\centering
\begin{tikzpicture}[scale=0.6]{thick}
\begin{scope}[every node/.style={circle,draw,minimum size=1pt,inner sep=2}]
	\node[ellipse,label={above:$v_1$}] (0) at (-3.5,0) {$(\Delta-2)^-$};
    \node[fill,label={above:$u_1$},label={below:$4$}] (1) at (0,0) {};
    \node[fill,label={above:$u$},label={below:$8$}] (2) at (2,0) {};
    \node[fill,label={above:$u_2$},label={below:$4$}] (3) at (4,0) {};
    \node[ellipse,label={above:$v_2$}] (4) at (7.5,0) {$(\Delta-2)^-$};
    
    \node[fill,label={above:$u_3$},label={below:$2$}] (2') at (3,2) {};
    \node[label={above:$v_3$}] (2'') at (4,4) {};
    
    \node[fill,label={above:$u_4$},label={below:$2$}] (3') at (1,2) {};
    \node[label={above:$v_4$}] (3'') at (0,4) {};
\end{scope}

\begin{scope}[every edge/.style={draw=black}]
    \path (0) edge (4);
    \path (2) edge (2'');
    \path (2) edge (3'');
\end{scope}
\end{tikzpicture}
\caption{A $(1,1,1,1)$-vertex that sees two $(\Delta-2)^-$-vertex at distance $2$.}
\end{figure}

\subsection{Discharging rules}

Since we have $\mad(G_2)<\frac{14}5$, we must have 
\begin{equation}\label{equation2}
\sum_{v\in V(G_2)} \left(5d(v)-14\right) < 0
\end{equation}

We assign to each vertex $v$ the charge $\mu(v)=5d(v)-14$. To prove the non-existence of $G_2$, we will redistribute the charges preserving their sum and obtaining a positive total charge, which will contradict \Cref{equation2}.

Observe that \Cref{2-path sponsor definition,111-path sponsor definition} also hold for $G_2$ thanks to \Cref{tree lemma2,tree lemma2bis}.

We apply the following discharging rules:

\medskip

\begin{itemize}
\item[\ru0] (see \Cref{r0 figure2}): Every $3^+$-vertex gives 2 to each 2-neighbor on an incident $1$-path.

\item[\ru1] (see \Cref{r1 figure2}): Let $u$ be incident to a $2$-path $P=uu_1u_2v$.
\begin{itemize}
\item[(i)] If $u$ is not $P$'s sponsor, then $u$ gives $\frac72$ to $u_1$.
\item[(ii)] If $u$ is $P$'s sponsor, then $u$ gives 4 to $u_1$ and $\frac12$ to $u_2$.
\end{itemize}

\item[\ru2] (see \Cref{r2 figure2}): 
\begin{itemize}
\item[(i)] Every $(\db{5}{7})$-vertex gives $1$ to each $3$-neighbor.
\item[(ii)] Every $8^+$-vertex gives $3$ to each $3$-neighbor.
\end{itemize}

\item[\ru3] (see \Cref{r3 figure2}): Let $uvw$ be a $1$-path.
\begin{itemize}
\item[(i)] If $u$ is a $\Delta$-vertex, $w$ is a $(1,1,1)$-vertex, and $u$ is $w$'s sponsor, then $u$ gives $2$ to $w$.
\item[(ii)] If $u$ is a $\Delta$-vertex, $w$ is a $(1,1,1)$-vertex, and $u$ is not $w$'s sponsor, then $u$ gives 1 to $w$.
\item[(iii)] If $u$ is a $\Delta$-vertex and $w$ is a $(1,1^-,0)$-vertex, then $u$ gives $\frac32$ to $w$.
\item[(iv)] If $u$ is a $9^+$-vertex and $w$ is a $4$-vertex, then $u$ gives $\frac23$ to $w$.
\end{itemize}
\end{itemize}

\begin{figure}[H]
\begin{center}
\begin{minipage}[b]{0.2\textwidth}
\begin{center}
\begin{tikzpicture}[baseline=(0.north)]
\begin{scope}[every node/.style={circle,thick,draw,minimum size=1pt,inner sep=2,font=\small}]
    \node[minimum width=23pt] (0) at (1,2) {$3^+$};
    \node[minimum width=23pt] (100) at (-1,2) {$3^+$};
    \node[fill] (3) at (0,2) {};
\end{scope}

\begin{scope}[every edge/.style={draw=black,thick}]
    \path (0) edge (100);
      \path[->] (0) edge[bend right] node[above] {2} (3);
      \path[->] (100) edge[bend left] node[above] {2} (3);
\end{scope}
\end{tikzpicture}
\caption{\ru0.}
\label{r0 figure2}
\end{center}
\end{minipage}
\begin{minipage}[b]{0.79\textwidth}
\centering
\begin{subfigure}[b]{0.39\textwidth}
\begin{center}
\begin{tikzpicture}[baseline=(3.north)]
\begin{scope}[every node/.style={circle,thick,draw,minimum size=1pt,inner sep=2,font=\small}]
    \node[minimum width=23pt,label={[label distance=-16pt]below:non-sponsor}] (0) at (-3,0) {$\Delta$};
    \node[fill] (2) at (-2,0) {};
    \node[fill] (3) at (-1,0) {};
    \node[minimum width=23pt,label={[label distance=-8pt]below:sponsor}] (4) at (0,0) {$\Delta$};
\end{scope}

\begin{scope}[every edge/.style={draw=black,thick}]
    \path (0) edge (4);    
    \path[->] (0) edge[bend left] node[above] {$\frac72$} (2);
\end{scope}
\end{tikzpicture}
\vspace*{-0.9cm}
\caption{\ }
\end{center}
\end{subfigure}
\begin{subfigure}[b]{0.39\textwidth}
\begin{center}
\begin{tikzpicture}[baseline=(3.north)]
\begin{scope}[every node/.style={circle,thick,draw,minimum size=1pt,inner sep=2,font=\small}]
    \node[minimum width=23pt,label={[label distance=-16pt]below:non-sponsor}] (0) at (-3,0) {$\Delta$};
    \node[fill] (2) at (-2,0) {};
    \node[fill] (3) at (-1,0) {};
    \node[minimum width=23pt,label={[label distance=-8pt]below:sponsor}] (4) at (0,0) {$\Delta$};
\end{scope}

\begin{scope}[every edge/.style={draw=black,thick}]
    \path (0) edge (4);    
    \path[->] (4) edge[bend left] node[below] {$4$} (3);
    \path[->] (4) edge[bend right] node[above] {$\frac12$} (2);
\end{scope}
\end{tikzpicture}
\vspace*{-0.9cm}
\caption{$2$-path sponsor.}
\end{center}
\end{subfigure}
\caption{\ru1.}
\label{r1 figure2}
\end{minipage}
\end{center}
\end{figure}

\begin{figure}[H]
\centering
\begin{subfigure}[b]{0.49\textwidth}
\begin{center}
\begin{tikzpicture}[baseline=(0.north)]
\begin{scope}[every node/.style={circle,thick,draw,minimum size=1pt,inner sep=2,font=\small}]
    \node[minimum width=23pt] (0) at (1,2) {$3$};
    \node[minimum width=23pt] (100) at (-0.5,2) {$\db57$};
\end{scope}

\begin{scope}[every edge/.style={draw=black,thick}]
    \path (0) edge (100);
      \path[->] (100) edge[bend left] node[above] {$1$} (0);
\end{scope}
\end{tikzpicture}
\caption{\ }
\end{center}
\end{subfigure}
\begin{subfigure}[b]{0.49\textwidth}
\begin{center}
\begin{tikzpicture}[baseline=(0.north)]
\begin{scope}[every node/.style={circle,thick,draw,minimum size=1pt,inner sep=2,font=\small}]
    \node[minimum width=23pt] (0) at (1,2) {$3$};
    \node[minimum width=23pt] (100) at (-0.5,2) {$8^+$};
\end{scope}

\begin{scope}[every edge/.style={draw=black,thick}]
    \path (0) edge (100);
      \path[->] (100) edge[bend left] node[above] {$3$} (0);
\end{scope}
\end{tikzpicture}
\caption{\ }
\end{center}
\end{subfigure}
\caption{\ru2.}
\label{r2 figure2}
\end{figure}

\begin{figure}[H]
\begin{center}
\begin{subfigure}[b]{0.49\textwidth}
\begin{center}
\begin{tikzpicture}[baseline=(3.north)]
\begin{scope}[every node/.style={circle,thick,draw,minimum size=1pt,inner sep=2,font=\small}]
    \node[minimum width=23pt,label={[label distance=-16pt]below:non-sponsor}] (1) at (-2,0) {$\Delta$};
    \node[fill] (1') at (-1,0) {};
    \node[fill,label={above:$w$}] (2) at (0,0) {};
    
    \node[fill,label={above:$v$}] (30) at (1,0.25) {};
    \node[minimum width=23pt,label={above:$u$}] (50) at (2,0.5) {$\Delta$};
    
    \node[fill] (31) at (1,-0.25) {};
    \node[minimum width=23pt] (51) at (2,-0.5) {$\Delta$};
    
    \node[draw=none] (100) at (3,-0.1) {sponsors};
\end{scope}

\begin{scope}[every edge/.style={draw=black,thick}]    
    \path (1) edge (2);
	
	\path (2) edge (50);
	\path (2) edge (51);
	
    \path[->] (50) edge[bend right=40] node[above] {$2$} (2);
    \path[->] (51) edge[bend left=40] node[below] {$2$} (2);
\end{scope}
\end{tikzpicture}
\vspace*{-0.8cm}
\caption{$(1,1,1)$-path sponsor.}
\end{center}
\end{subfigure}
\begin{subfigure}[b]{0.49\textwidth}
\begin{center}
\begin{tikzpicture}[baseline=(3.north)]
\begin{scope}[every node/.style={circle,thick,draw,minimum size=1pt,inner sep=2,font=\small}]
    \node[minimum width=23pt,label={[label distance=-16pt]below:non-sponsor},label={above:$u$}] (1) at (-2,0) {$\Delta$};
    \node[fill,label={above:$v$}] (1') at (-1,0) {};
    \node[fill,label={above:$w$}] (2) at (0,0) {};
    
    \node[fill] (30) at (1,0.25) {};
    \node[minimum width=23pt] (50) at (2,0.5) {$\Delta$};
    
    \node[fill] (31) at (1,-0.25) {};
    \node[minimum width=23pt] (51) at (2,-0.5) {$\Delta$};
    
    \node[draw=none] (100) at (3,-0.1) {sponsors};
\end{scope}

\begin{scope}[every edge/.style={draw=black,thick}]    
    \path (1) edge (2);
	
	\path (2) edge (50);
	\path (2) edge (51);
	
    \path[->] (1) edge[bend left=40] node[above] {$1$} (2);
\end{scope}
\end{tikzpicture}
\vspace*{-0.8cm}
\caption{}
\end{center}
\end{subfigure}
\begin{subfigure}[b]{0.49\textwidth}
\begin{center}
\begin{tikzpicture}[baseline=(3.north)]
\begin{scope}[every node/.style={circle,thick,draw,minimum size=1pt,inner sep=2,font=\small}]
    \node[minimum width=23pt,label={above:$u$}] (1) at (-2,0) {$\Delta$};
    \node[fill,label={above:$v$}] (1') at (-1,0) {};
    \node[fill,label={above:$w$}] (2) at (0,0) {};
    
    \node[minimum width=23pt] (50) at (2,0.5) {$2^+$};
    
    \node[minimum width=23pt] (51) at (2,-0.5) {$3^+$};
\end{scope}

\begin{scope}[every edge/.style={draw=black,thick}]    
    \path (1) edge (2);
	
	\path (2) edge (50);
	\path (2) edge (51);
	
    \path[->] (1) edge[bend left=40] node[above] {$\frac32$} (2);
\end{scope}
\end{tikzpicture}
\caption{\ }
\end{center}
\end{subfigure}
\begin{subfigure}[b]{0.49\textwidth}
\centering
\begin{tikzpicture}[baseline=(0.north)]
\begin{scope}[every node/.style={circle,thick,draw,minimum size=1pt,inner sep=2,font=\small}]
    \node[minimum width=23pt,label={above:$w$}] (0) at (1,2) {$4$};
    \node[minimum width=23pt,label={above:$u$}] (100) at (-1,2) {$9^+$};
    \node[fill,label={above:$v$}] (3) at (0,2) {};
\end{scope}

\begin{scope}[every edge/.style={draw=black,thick}]
    \path (0) edge (100);
      \path[->] (100) edge[bend left=40] node[above] {$\frac23$} (0);
\end{scope}
\end{tikzpicture}
\caption{\ }
\end{subfigure}
\caption{\ru3.}
\label{r3 figure2}
\end{center}
\end{figure}

\subsection{Verifying that charges on each vertex are non-negative} \label{verification2}

Let $\mu^*$ be the assigned charges after the discharging procedure. In what follows, we prove that: $$\forall u \in V(G_2), \mu^*(u)\ge 0.$$

Let $u\in V(G_2)$.

\textbf{Case 1: }If $d(u)=2$, then recall that $\mu(u)=5\cdot 2 - 14 = - 4$.\\
Recall that there exists no $3^+$-path due to \Cref{3-path lemma2}. So, $u$ must lie on a $1$-path or a $2$-path.

If $u$ is on a $1$-path, then it has two $3^+$-neighbors which give it 2 each by \ru0. Thus,
$$ \mu^*(u)= -4 + 2\cdot 2 = 0.$$

If $u$ is on a $2$-path, then $u$ receives 4 from an adjacent sponsor by \ru1(ii), or it receives $\frac72+\frac12 = 4$ from an adjacent non-sponsor and a distance 2 sponsor respectively by \ru1(i) and \ru1(ii). Thus,
$$ \mu^*(u)= -4 + 4 = 0.$$

\textbf{Case 2: }If $d(u)=3$, then recall that $\mu(u)=5\cdot 3 - 14 = 1$.\\
Observe that $u$ only gives charge away by \ru0 (charge 2 to each $2$-neighbor).

If $u$ is a $(1,1,1)$-vertex, then the other endvertices of the 1-paths incident to $u$ are all $\Delta$-vertices due to \Cref{111-10}. As a result, $u$ receives $2$ from each of its two sponsors and $1$ from the non-sponsor $\Delta$-vertex by \ru3(i) and \ru3(ii). Hence,
$$ \mu^*(u)= 1 -3\cdot 2 + 2\cdot 2 + 1 = 0.$$

If $u$ is a  $(1,1,0)$-vertex with a $8^+$-neighbor, then it receives $3$ from its $8^+$-neighbor by \ru2(ii). Thus,
$$ \mu^*(u)= 1 -2\cdot 2 + 3 = 0.$$

If $u$ is a  $(1,1,0)$-vertex with an $7^-$-neighbor ($7\leq \Delta-3$ since $\Delta\geq 10$), then it receives $\frac32$ by \ru3(iii) from each of the other endvertices of its incident $1$-paths due to \Cref{8-011-10}. Thus,
$$ \mu^*(u)= 1 -2\cdot 2 + 2\cdot \frac32 = 0.$$

If $u$ is a $(1,0,0)$-vertex with a $5^+$-neighbor, then it receives at least $1$ from that neighbor by \ru2. Thus,
$$ \mu^*(u)\geq 1 -2 + 1 = 0.$$

If $u$ is a $(1,0,0)$-vertex with two $(\db34)$-neighbors, then it receives $\frac32$ by \ru3(i) from the other endvertex of its incident 1-path due to \Cref{4-00(-4)1-10}. So,
$$ \mu^*(u)= 1 -2 + \frac32 = \frac12.$$

If $u$ is a $(0,0,0)$-vertex, then
$$ \mu^*(u) = \mu(u) = 1.$$

\textbf{Case 3: }If $d(u)=4$, then recall that $\mu(u)=5\cdot 4 - 14 = 6$.\\
Observe that $u$ only gives charge by \ru0 (charge 2 to each $2$-neighbor).

If $u$ is a $(1,1,1,1)$-vertex, then at least three of the four other endvertices of the 1-paths incident to $u$ are $(\Delta-1)^+$-vertices (which are $9^+$-vertices since $\Delta\geq 10$) due to \Cref{1111}. As a result, $u$ receives $\frac23$ from each of the $9^+$-endvertex by \ru3(iv). Hence,
$$ \mu^*(u)\geq 6 -4\cdot 2 + 3\cdot \frac23 = 0.$$

If $u$ is a $(1^-,1^-,1^-,0)$-vertex, then
$$ \mu^*(u) \geq 6 - 3\cdot 2 = 0.$$

\textbf{Case 3: }If $5\leq d(u)\leq 7$, then $u$ can give 2 to each $2$-neighbor by \ru0 or 1 to each $3$-neighbor by \ru2(i). Thus, at worst we get
$$ \mu^*(u)\geq 5d(u) - 14 - 2d(u) \geq 3\cdot 5 - 14 = 1.$$

\textbf{Case 4: }If $8\leq d(u)\leq \Delta-1$, then $u$ gives at most $3$ along each incident path by \ru2(ii). Thus, at worst we get
$$ \mu^*(u)\geq 5d(u) - 14 - 3d(u) \geq 2\cdot 8 - 14 = 2.$$

\textbf{Case 5: }If $d(u)=\Delta$, then we distinguish the following cases.
\begin{itemize}
\item If $u$ is neither a $2$-path sponsor nor a $(1,1,1)$-path sponsor, then observe that $u$ gives away at most $\frac72$ along an incident path by \ru1(i) or a combination of \ru0 and \ru3(iii). So at worst, $$\mu^*(u)\geq 5\Delta - 14 - \frac72\Delta \geq \frac32\cdot 10 - 14 =1.$$

\item If $u$ is a $2$-path sponsor but not a $(1,1,1)$-path sponsor, then $u$ gives $4+\frac12=\frac92$ to its unique sponsored incident $2$-path by \ru1(ii). For the other incident paths, it gives at most $\frac72$ like above. So,
$$ \mu^*(u)\geq 5\Delta - 14 - \frac92 - \frac72(\Delta-1) \geq \frac32\cdot 10 - 14 - \frac92 + \frac72 = 0.$$

\item If $u$ is a $(1,1,1)$-path sponsor but not a $2$-path sponsor, then $u$ gives $2+2=4$ to the unique incident $(1,1,1)$-path containing its assigned $(1,1,1)$-vertex $v$: $1$ to the $2$-neighbor by \ru0 and $1$ to $v$ by \ru3(i). Once again, $u$ gives at most $\frac72$ to the other incident paths. So,
$$ \mu^*(u)\geq 5\Delta - 14 - 4 - \frac72(\Delta-1) \geq \frac32\cdot 10 - 14 - 4 + \frac72 = \frac12.$$

\item If $u$ is both a $2$-path sponsor and a $(1,1,1)$-path sponsor, then $u$ gives $\frac92$ to its unique sponsored $2$-path and $4$ to its unique assigned $(1,1,1)$-vertex like above. 

Now, let us consider the other $\Delta-2$ paths incident to $u$. Observe that when $u$ gives $\frac72$ along an incident path either by \ru1(i) or by a combination of \ru0 and \ru3(iii), that path must be a $1^+$-path where the vertex at distance $2$ from $u$ is a $3^-$-vertex. Due to \Cref{weird cases lemma2}, $u$ never has to give $\frac72$ to each of the $\Delta-2$ paths. As a result, there exists one path to which $u$ gives at most 3. So at worst,
$$ \mu^*(u)\geq 5\Delta - 14 - \frac92 - 4 - 3 - \frac72(\Delta-3) \geq \frac32\cdot 10 - 14 -\frac92 - 4 - 3 + \frac{21}2 = 0. $$ 
\end{itemize}

We obtain a non-negative amount of charge on each vertex, which is impossible since the total amount of charge is negative. As such, $G_2$ cannot exist. That concludes the proof of \Cref{main theorem2}.

\section*{Acknowledgements}
This work was partially supported by the grant HOSIGRA funded by the French National Research
Agency (ANR, Agence Nationale de la Recherche) under the contract number ANR-17-CE40-0022.

\bibliographystyle{plain}

\end{document}